\documentclass{amsart}

\usepackage{amssymb} \usepackage{amsfonts} \usepackage{amsmath}
\usepackage{amsthm}
\usepackage{epsfig} 
\usepackage{color}
\usepackage{amscd}
\usepackage{comment}
\usepackage[all]{xy}
\usepackage[english]{babel}
\usepackage{hyperref}
\usepackage{amssymb, amsfonts, amsmath}
\usepackage{pgfplots}
\usepackage{comment} 
\usepackage{mathtools}
\usepackage{wrapfig}
\usepackage{tikz}
\usepackage{comment}
\usepackage{mathrsfs}
\usepackage{tabularx, hyperref}
\usepackage{amssymb} \usepackage{amsfonts} \usepackage{amsmath}
\usepackage{amsthm} \usepackage{epsfig, subfig}
\usepackage{ amscd, amsxtra, latexsym}
\usepackage{epsfig,  graphicx, psfrag}
\usepackage[all]{xy}
\usepackage{caption}
\usepackage{enumerate}
\usepackage{color}
\usepackage{yhmath}
\usepackage{tikz-cd}

\newtheorem{lemma}{Lemma}[section]
\newtheorem{thm}[lemma]{Theorem}
\newtheorem{prop}[lemma]{Proposition}
\newtheorem{cor}[lemma]{Corollary}

\newtheorem*{prop*}{Proposition}
\newtheorem{prop_intro}{Proposition}
\newtheorem{conj_intro}[prop_intro]{Conjecture}

\newtheorem{thm_intro}[prop_intro]{Theorem}

\theoremstyle{definition}

\newtheorem{defn}[lemma]{Definition}

\newtheorem{rem}[lemma]{Remark}

\theoremstyle{definition}

\newtheoremstyle{citing}
{3pt}
{3pt}
{\itshape}
{}
{\bfseries}
{.}
{.5em}
{\thmnote{#3}}
\theoremstyle{citing}

\newcommand{\matN}{\ensuremath {\mathbb{N}}}
\newcommand{\matR} {\ensuremath {\mathbb{R}}}

\newcommand{\calT} {\ensuremath {\mathcal{T}}}

\newcommand{\ori}{{\textrm{Or}}}

\newcommand{\str} {\ensuremath {{\mathrm{str}}}}

\newcommand{\omeo}{\operatorname{Homeo}^+(S^1)}

\DeclareMathOperator{\isom}{Isom}

\newcommand{\commento}[1]{}

\author{Roberto Frigerio}
\address{Dipartimento di Matematica, Largo Pontecorvo 5, 56127 Pisa, Italy}
\email{roberto.frigerio@unipi.it}

\author{Ervin Had\v{z}iosmanovi\'c}
\address{Scuola Normale Superiore, 
Piazza dei Cavalieri 7, 56126, Pisa, Italy}
\email{ervin.hadziosmanovic@sns.it}

\title[]{The bounded area class\\ of negatively curved surfaces}


\begin{document}

\begin{abstract}
Let $S$ be an oriented surface, possibly of infinite type, endowed with a complete Riemannian metric with pinched negative curvature. We prove that the 
area form defines a non-trivial class in the second bounded cohomology group of $S$, unless $S$ is diffeomorphic to the disc or the cylinder. This is in sharp contrast with the $n$-dimensional case, $n>2$, where the non-triviality of the volume form in bounded cohomology is related to the Cheeger constant of the manifold. 
We also discuss how the bounded area class  depends on the metric: we prove that, for compact surfaces, it recognizes constant curvature metrics among pinched negatively curved ones, while, even in the case of surfaces of infinite type, it 
does not distinguish non-isometric hyperbolic structures. More precisely, when $S$ is compact we show that, among the negatively curved structures of fixed area, the ones with constant curvature provide the unique minimizers for the norm of the bounded area class.
\end{abstract}

\maketitle

\section*{Introduction}
Let $S$ be an oriented connected surface, possibly of infinite type, endowed with a complete Riemannian metric $g$ with pinched negative curvature.
Negative curvature allows one to define a \emph{straightening procedure} 
which deforms any singular simplex with values in $S$ into a geodesic one.
When the curvature is bounded away from 0, 
straight simplices have a uniformly  bounded area. Hence, integration of the area form of $g$ on straight 2-dimensional simplices defines a singular cocycle
${\omega}_g\in C^2(S)$ which is \emph{bounded}, i.e.~which attains uniformly bounded values on singular simplices (see Section~\ref{prelim:sec} for the details). Therefore, the bounded cocycle $\omega_g$ defines  the \emph{bounded area class} $[\omega_g]$ in the second bounded cohomology  group with real coefficients $H^2_b(S)$.  
In Section~\ref{nontrivial:sec}  we prove the following:

\begin{thm_intro}\label{main:surface:thm}
Let  $g$ be a complete Riemannian metric with pinched negative curvature  on the oriented and connected surface $S$. Then
$[\omega_g]\neq 0$ in $H^2_b(S)$, unless $S$ is diffeomorphic to $\mathbb{R}^2$ or to $S^1\times \mathbb{R}$ (in which cases
$[\omega_g]= 0$).
\end{thm_intro}

We stress here that in our Theorem~\ref{main:surface:thm} we do not make any finiteness assumption on $S$, which may well be of infinite type.

Theorem~\ref{main:surface:thm} is in sharp contrast with what happens in higher dimensions. For example, while it is well known that 
the bounded volume class of a complete finite-volume pinched negatively curved $n$-manifold $M$ is non-trivial in $H^n_b(M)$, things get more complicated
when considering manifolds with infinite volume.
In~\cite{soma1997bounded}, Soma proved that the bounded volume class of an infinite volume hyperbolic $3$-manifold with finitely generated fundamental group vanishes if and only if the manifold is \emph{geometrically finite} (he actually considered topologically tame hyperbolic $3$-manifolds,
but by the solution of Marden's Tameness Conjecture by Agol \cite{agol2004tameness} and Calegari-Gabai \cite{calegari2006shrinkwrapping}, this is now known to be equivalent to considering 
manifolds with a finitely generated fundamental group). This implies in turn (thanks to results by Hamenst\"adt \cite{hamenstadt2004small}, Ledoux~\cite{ledoux1994simple}
and Bonahon and Canary~\cite{bonahon1986bouts,Canary}) that the bounded volume class of an infinite volume hyperbolic $3$-manifold $M$ with finitely generated fundamental group vanishes if and only if the Cheeger constant $h(M)$ of $M$ vanishes
(roughly speaking, the vanishing of $h(M)$ means that compact domains in $M$ do \emph{not} satisfy a linear isoperimetric inequality).

In more recent years, Kim and Kim extended this last statement to the context of 
pinched negatively curved $3$-manifolds of bounded geometry~\cite[Theorem 1.4]{kim2015bounded} and of
higher-dimensional locally symmetric spaces of rank one~\cite[Theorem 1.2]{kim2015bounded}.
These results lead them to the formulation of the following:
\begin{conj_intro}[\cite{kim2015bounded}]\label{conjecture}
	Let $(M,g)$ be a complete pinched negatively curved manifold of infinite volume and dimension $n\ge3$. Then the bounded volume class $[\omega_g] \in H^{n}_b(M)$ vanishes if and only if $h(M)>0$.
\end{conj_intro}

The second author has recently proved this conjecture  for  negatively curved Riemannian manifolds of bounded geometry with curvature bounded away from 0~\cite[Theorem 1]{hadziosmanovic2025bounded}, but the general case is still open.

Our Theorem~\ref{main:surface:thm} shows that no analogue of these results holds in dimension 2.

\subsection*{Dependence on the metric: hyperbolic surfaces}
An interesting topic is
the dependence of the bounded volume class $[\omega_g]$ on the Riemannian metric $g$.
The case of hyperbolic $3$-dimensional manifolds has been extensively studied.
Generalising the result of Soma mentioned above, Farre \cite{farre2018bounded} proved that if two geometrically infinite hyperbolic structures on the $3$-manifold $M$ have different ending laminations, then the induced bounded volume classes are linearly independent in $H^3_b(M)$ (see also \cite{farre2020relations}). 
Thus, by varying the hyperbolic metrics on $M$ one can obtain an infinite dimensional subspace of $H^3_b(M)$ spanned by bounded volume classes. 
The situation is different in the $2$-dimensional case. In fact, we provide here a complete proof of the following result: 

\begin{thm_intro}\label{hyperbolic:equal}
Let $S$ be a surface (possibly of infinite type), and 
let $g_1$ and $g_2$ be two complete hyperbolic structures on $S$. Then  $[\omega_{g_1}]=[\omega_{g_2}]$
in $H^2_b(S)$.
\end{thm_intro}

If $S$ is hyperbolic, the area of every ideal triangle in the universal covering $\widetilde{S}$ is equal to $\pi$. This easily implies that the bounded area class is equal (up to a universal scaling factor) to the 
real bounded Euler class of the action of $\pi_1(S)$ on the circle $\partial\widetilde{S}=\partial \mathbb{H}^2$. 
A fundamental result by Ghys now ensures that two circle actions
of the same group have the same integral bounded Euler class
 if and only if they are semi-conjugate. Therefore, Theorem~\ref{hyperbolic:equal} may be deduced from the following:

\begin{thm_intro}\label{semi-conj:thm}
Let $S$ be an orientable surface (possibly of infinite type), 
let $g_1$ and $g_2$ be two complete pinched negatively curved structures on $S$, and let $\rho_i\colon \pi_1(S)\to \operatorname{Homeo}^+(\partial\widetilde{S}_i)$, $i=1,2$,
be the induced circle actions, where $ \widetilde{S}_i$ 
is the Riemannian universal covering of $(S,g_i)$. Then $\rho_1$ and $\rho_2$ are semi-conjugate.
\end{thm_intro}

Theorem~\ref{semi-conj:thm} (and thus Theorem ~\ref{hyperbolic:equal}) is known  in the case when $S$ is a surface of finite type (see~\cite[Corollaries 4.3 and 4.5]{BIWsurvey} and the proof of Theorem 3 in~\cite{BIW}). However, for surfaces of infinite type 
there does not seem any available reference for it. We provide a complete proof of Theorem~\ref{semi-conj:thm}
in
Section~\ref{semi-conj:sec}, while in Section~\ref{euler:sec} we show how to deduce Theorem~\ref{hyperbolic:equal} from Theorem~\ref{semi-conj:thm}.

\subsection*{Dependence on the metric: surfaces with non-constant curvature}
Theorem~\ref{hyperbolic:equal} shows that the bounded area class cannot distinguish non-isometric hyperbolic structures on the same surface. 
However, at least for closed surfaces, it recognizes hyperbolic structures among all pinched negatively curved structures:

\begin{thm_intro}\label{recognizing:hyperbolic}
Let $S$ be a closed oriented surface, and 
$g_1$ and $g_2$ be  negatively curved Riemannian metrics on $S$ such that 
$[\omega_{g_1}]=[\omega_{g_2}]$ in $H^2_b(S)$. Then $g_1$ is hyperbolic if and only if $g_2$ is hyperbolic.
\end{thm_intro}

Building on Theorem~\ref{recognizing:hyperbolic} and on an extremality property for the real bounded Euler class (see~\cite[Theorem 4.14]{BIWsurvey}) we also obtain 
the following:

\begin{thm_intro}\label{extremal:thm}
Let $S$ be a closed oriented surface, and let $g$ be a negatively curved metric on $S$ whose area is equal to $2\pi |\chi(S)|$. Then
$\|\omega_g\|_\infty\geq \pi$,  and the equality
$$
\|\omega_g\|_\infty=\pi
$$
holds if and only if $g$ is hyperbolic.
\end{thm_intro}

Theorem~\ref{extremal:thm} can be viewed as a bounded-cohomological analogue of classical geometric rigidity theorems for negatively curved manifolds, such
as Katok's entropy rigidity \cite{katok1982}
or, in higher dimensions, Besson--Courtois--Gallot's minimal volume entropy rigidity \cite{BCG1995}. Katok's entropy rigidity for surfaces characterizes metrics of constant negative curvature 
as minimizers for the topological entropy,
while
our result characterizes them as minimizers for  the  norm of the bounded area form.

At least when $S$ is closed, while  all the hyperbolic metrics on $S$ define
the same bounded area class, 
the bounded area classes associated to non-hyperbolic metrics span a
large subspace of $H^2_b(S)$: 

\begin{thm_intro}\label{MCG:action}
Let $(S,g)$ be a closed negatively curved surface, and suppose that the curvature of $g$ is not constant. Then
the orbit of $[{\omega}_g]$ in $H^2_b(S)$ via the action of the mapping class group of $S$ spans an infinite dimensional
subspace of $H^2_b(S)$. In particular, the subspace of $H^2_b(S)$ spanned by the bounded area classes of all negatively curved metrics on $S$ is infinite dimensional.
\end{thm_intro}

Theorems~\ref{recognizing:hyperbolic}, \ref{extremal:thm} and~\ref{MCG:action} are proved in Section~\ref{nonhyp:sec}.

\subsection*{Acknowledgements}  Both authors are partially
supported by INdAM through GNSAGA.

\section{The bounded area class}\label{prelim:sec}

\subsection{Bounded cohomology of spaces} 
Let $X$ be a topological space. 
In this paper we consider only (co)homology with real coefficients, hence
we denote by $(C_\bullet(X),\delta^\bullet)$ (resp.~$(C^\bullet(X),\delta^\bullet)$) the complex of singular chains (resp.~cochains) on $X$ with real coefficients, and by $H_\bullet(X)$ (resp.~$H^\bullet(X)$) the corresponding (co)homology modules. 
Let $S_n(X)$ denote the set of singular $n$-simplices with values in $X$.
We define an $\ell^\infty$-norm on cochains as follows:
if $\varphi\in C^n(X)$, then
$$
\|\varphi\|_\infty=\sup \{|\varphi(\sigma)|\, ,\ \sigma\in S_n(X)\}\ \in\ [0,+\infty]\ .
$$
 We say that $\varphi\in C^n(X)$ is \emph{bounded} if $\|\varphi\|_\infty<\infty$, and we denote by $C^n_b(X)$ the subspace of bounded $n$-cochains. 
 It is easy to check that, for every $\varphi\in C^n(X)$, one has $\|\delta^n \varphi\|_\infty\leq (n+1) \|\varphi\|_\infty$, hence
 $C^\bullet_b(X)$ is a subcomplex of $C^\bullet(X)$, whose cohomology is called the \emph{bounded cohomology} of $X$, and is denoted by
 $H^\bullet_b(X)$. The norm $\|\cdot \|_\infty$ on the space of cocycles in $C_b^n(X)$ descends to a quotient seminorm on $H^n_b(X)$, which will still be denoted by
 $\|\cdot \|_\infty$.
 
The $\ell^\infty$-norm on $C_b^n(X)$ is dual to the $\ell^1$-norm on $C_n(X)$, which is defined by
$$
\left\|\sum_{\sigma \in S_n(X)} a_\sigma \cdot \sigma\right\|_1=\sum_{\sigma \in S_n(X)} |a_\sigma|\ .
$$
The restriction of this norm to the subspace of cycles descends to a seminorm on $H_n(X)$, which will stil be denoted by $\|\cdot\|_1$. 
The obvious pairing between $C^n_b(X)$ and $C_n(X)$ induces a pairing
$$
\langle \cdot, \cdot, \rangle\colon H^n_b(X)\to H_n(X)\to \mathbb{R}\ ,
$$
 which is usually called \emph{Kronecker pairing}.

When $X$ is a closed oriented manifold, we denote by $[X]\in H_n(X)$ the \emph{real fundamental class} of $X$, i.e.~the image of the integral fundamental class of $X$ via the change of coefficient morphism. 

 \subsection{Straightening in negative curvature}
Let $S$ be a connected oriented surface,  endowed with a complete pinched negatively curved  Riemannian structure $g$.
One can \emph{straighten} simplices in 
the Riemannian universal covering $(\widetilde{S},\widetilde{g})$ of $(S,g)$ as follows: 
if $\sigma\colon \Delta^0\to \widetilde{S}$ is a $0$-simplex, then one simply sets $\widetilde{\str}_0(\sigma)=\sigma$;
if $\sigma\colon [0,1]\to \widetilde{S}$ is a $1$-simplex, then $\widetilde{\str}_1(\sigma)\colon  [0,1]\to \widetilde{S}$
is the constant speed parametrization of the geodesic starting at $\sigma(0)$ and ending at $\sigma(1)$; 
if $v_0,v_1,v_2$ are the vertices of the standard simplex $\Delta^2$ and $\sigma\colon \Delta^2\to \widetilde{S}$ is a $2$-simplex,
then 
the restriction of $\widetilde{\str}_2 (\sigma)$ to the edge $[v_0,v_1]\subseteq \Delta^2$ with endpoints $v_0,v_1$
is equal to $\widetilde{\str}_1(\sigma|_{[v_0,v_1]})$, and the restriction of $\widetilde{\str}_2 (\sigma)$ to
any segment $[z,v_2]\subseteq \Delta^2$, $z\in [v_0,v_1]$, is the constant speed
parameterization of the geodesic joining $\widetilde{\str}_{1}(\sigma|_{[v_0,v_1]}(z))$ to $\sigma(v_2)$.
In other words, $\widetilde{\str}_2 (\sigma)$ is the geodesic cone  from the vertex $\sigma(v_2)$ over the geodesic joining
$\sigma(v_0)$ to $\sigma(v_1)$.
In this construction,
negative curvature plays an essential role in ensuring that the geodesic joining two points in $\widetilde{S}$ is unique and smoothly depends on its endpoints, so that the coning procedure indeed provides a well-defined smooth map (see e.g.~\cite[Proposition 2.4]{lohsauer}).  
 
 For any singular simplex $\sigma\colon \Delta^k\to S$, $k=0,1,2$, the \emph{straightening} ${\str}_k(\sigma)\colon \Delta^k\to {S}$ of ${\sigma}$
is defined as the projection in $S$ of $\str_k(\widetilde{\sigma})$, where $\widetilde{\sigma}$ is a lift of $\sigma$ to $\widetilde{S}$. It is not difficult to show
that  $\str_k(\widetilde{\sigma})$ does not depend on the choice of the lift $\widetilde{\sigma}$. The straightening procedure can be defined similarly (by inductively coning over
the last vertex of the standard simplex) for simplices of any dimension, thus defining a chain map
$$
\str_\bullet\colon C_\bullet (S)\to C_\bullet(S)
$$
 chain-homotopic to the identity (see e.g.~\cite[Section 8.7]{Frigeriobook}). Chains (or simplices) lying in the image of $\str_\bullet$ are called \emph{straight}.

Recall now that $S$ is oriented, and  denote by $dA_g$ the area form of $S$ associated to the Riemannian metric $g$.
A direct consequence of the Gauss-Bonnet formula is that the supremum of the areas of straight triangles 
in a surface with negative curvature bounded away from 0 
is finite. Therefore, the singular cochain
$$
\omega_g \in C^2(S)\, ,\qquad \omega_g(\sigma)=\int_{\str_2(\sigma)}dA_g
$$
is bounded. Moreover, Stokes' Theorem and the fact that $\str_\bullet$ is a chain map imply that the cochain $\omega_g$ is a cocycle.
We can thus denote by $[\omega_g]\in H^2_b(S)$ the bounded cohomology class represented by $\omega_g$.

As already observed in~\cite{barge1988surfaces}, the bounded class $[\omega_g]$ 
can also be described by integration on ideal triangles with vertices on the boundary at infinity $\partial \widetilde{S}$, rather than on straight triangles with vertices inside $\widetilde{S}$.
In order to make this statement precise, one needs to switch from the bounded cohomology of $S$ to the bounded cohomology of its fundamental group.

\subsection{Bounded cohomology of groups}\label{group:sub}
Let $G$  be a discrete group.
The \emph{bound\-ed cohomology} of $G$ (with real coefficients), denoted by $H^\bullet_b(G)$, is defined as the cohomology of the following complex of vector spaces
\[
0 \rightarrow C^0_{b}(G)^G \rightarrow C^1_{b}(G)^G\rightarrow C^2_{b}(G)^G \rightarrow \cdots\, ,
\]
where $C^n_{b}(G)^G$ denotes the space of bounded $G$-invariant maps from $G^{n+1}$ to $\matR$, the differential maps are defined by $$(\delta \varphi)(\gamma_0,\dots,\gamma_n) = \sum_{i=0}^n(-1)^i\varphi(\dots,\widehat{\gamma_i},\dots)\, ,$$
and the action of $G$ on $C^n_b(G)$ is given by $(\gamma\cdot \varphi)(\gamma_0,\dots,\gamma_n)=\varphi(\gamma^{-1}\gamma_0,\dots, \gamma^{-1}\gamma_n)$.
The $\ell^\infty$-norm  on the space of cocycles in $C^n_{b}(G)^G$ descends to a quotient seminorm on $H^n_b(G)$, which will be denoted by
 $\|\cdot \|_\infty$.

A celebrated theorem due to Gromov ensures that the bounded cohomology of a topological space is isometrically isomorphic to the bounded cohomology of its fundamental group.
In our case of interest, i.e.~ when $S=\widetilde{S}/\Gamma$, $\Gamma\cong \pi_1(S)$, is a complete pinched negatively curved surface, Gromov's isomorphism in degree 2
$$
\theta\colon H^2_b(S)\to H^2_b(\Gamma)
$$
admits a very explicit description, which we will 
exploit in our study of the bounded area class.

Let  $\calT$ be the set of elements $(\xi_0,\xi_1,\xi_2)\in (\partial\widetilde{S})^3$ such that $\xi_i\neq \xi_j$ for $i\neq j$. 
Thus, $\mathcal{T}$ parametrizes the set of ideal triangles
with ordered vertices in $\widetilde{S}$. For every $(\xi_0,\xi_1,\xi_2)\in \mathcal{T}$, we denote by  $T(\xi_0,\xi_1,\xi_2)$ an
\emph{oriented} parametrization  of the ideal triangle $T$ with vertices $\xi_0,\xi_1,\xi_2$, i.e.~the restriction to $\Delta^2\setminus\{v_0,v_1,v_2\}$ of a parametrization $\sigma\colon \Delta^2 \to \overline{T}\subseteq \widetilde{S}\cup \partial\widetilde{S}$ 
sending $v_i$ to $\xi_i$ for $i=0,1,2$, where $v_0,v_1,v_2$ are the vertices of $\Delta^2$.

Let now $\xi\in\partial\widetilde{S}$ be a fixed basepoint. We consider the cochain
$\omega_{g,\xi}\in C^2_b(\Gamma)^\Gamma$ defined  by
$$
\omega_{g,\xi}(\gamma_0,\gamma_1,\gamma_2)=\left\{\begin{array}{ll} \int_{T(\gamma_0 {\xi}, \gamma_1 {\xi}, \gamma_2 {\xi})}dA_{\widetilde{g}} \quad & \text{if}\ 
	(\gamma_0 {\xi}, \gamma_1 {\xi}, \gamma_2 {\xi})\in\mathcal{T}\\
	0 \quad &\text{otherwise}\ ,\end{array}\right.
$$
where $dA_{\widetilde{g}}$ is the area form on $\widetilde{S}$ (i.e., the pull-back to $\widetilde{S}$ of $dA_g$).
In other words, $\omega_{g,\xi}(\gamma_0,\gamma_1,\gamma_2)$ is the signed area of the ideal triangle $T(\gamma_0 {\xi}, \gamma_1 {\xi}, \gamma_2 {\xi})$,
where the sign depends on the cyclic ordering of $\gamma_0 {\xi}, \gamma_1 {\xi}, \gamma_2 {\xi}$ on $\partial \widetilde{S}$.
The following result was proved by Barge and Ghys, and allows us
to represent the class $\theta([\omega_g])$ ``on the boundary'' of $\widetilde{S}$.

\begin{lemma}[{\cite[Lemma 3.10]{barge1988surfaces}}]\label{boundaryrep:lemma}
	For every $\xi\in\partial\widetilde{S}$, the group cochain $\omega_{g,\xi}$ is a bounded cocycle representing
	the class $\theta([\omega_g])\in H^2_b(\Gamma)$.
\end{lemma}
\qed

\subsection{Bounded cocycles as measurable functions}\label{sec: bc as meas fun}
We are now going to exploit some fundamental results concerning the use of amenable actions
for the computation of bounded cohomology.  Let $S$ and $\Gamma$ be as above. 
We endow $\partial\widetilde{S}$ with the measure class $\mu$ 
of the  visual measure associated to any basepoint in $\widetilde{S}$.

It is well known (and very easy to check using e.g.~that every element of $\Gamma$ acts as a diffeomorphism of $\partial \widetilde{S}$) 
that $(\partial\widetilde{S},\mu)$ is a regular $\Gamma$-space according to~\cite[Definition 2.1.1]{Monod}.
Moreover, the action of $\Gamma$ on $\partial \widetilde{S}$ is \emph{Zimmer amenable} (we refer to~\cite[Chapter 4]{zimmer} for the definition of this notion): 
in fact, since $\Gamma$ is a discrete subgroup of $\operatorname{Isom}^+(\widetilde{S})$, {\cite[Theorem 3.1]{spatzier} readily implies the following:
	
	\begin{lemma}
		\label{amenability}
		The space $(\partial\widetilde{S},\mu)$ is an amenable regular $\Gamma$-space.
	\end{lemma}

	The fundamental r\^ole of amenable actions in the theory of bounded cohomology is described by Theorem~\ref{amenable_resolution} below, which allows us to compute
	the bounded cohomology of $\Gamma$ 
	via the cochain complex 
	\begin{equation}\label{resolution}
		0\rightarrow L^\infty_{\rm alt}(\partial\widetilde{S})^\Gamma\rightarrow  L^\infty_{\rm alt}((\partial\widetilde{S})^2)^\Gamma\rightarrow L^\infty_{\rm alt}((\partial\widetilde{S})^3)^\Gamma\rightarrow\cdots 
	\end{equation}
	of $\Gamma$-invariant alternating measurable bounded functions on $(\partial\widetilde{S})^{n+1}$ up to equality almost everywhere, endowed with the usual differential
	
	\[(\delta \varphi)(x_0,\dots,x_n) = \sum_{i=0}^n(-1)^i \varphi(\dots,\widehat{x_i},\dots)\ .\]
	Recall that a cochain $\varphi\colon (\partial\widetilde{S})^{n+1}\to\mathbb{R}$ is \emph{alternating} if $$\varphi(x_{\tau(0)},\dots,x_{\tau(n)})=\varepsilon(\sigma)\varphi(x_0,\dots,x_n)$$ for every
	permutation $\tau$ of $\{0,\dots,n\}$, where $\varepsilon(\tau)$ is the sign of $\tau$.
	
	Since $\partial\widetilde{S}$ is a regular amenable $\Gamma$-space, \cite[Theorem 2]{burger2001continuous} (or~\cite[Theorem 7.5.3]{Monod})
	implies the following:
	
	\begin{thm}\label{amenable_resolution}
		The bounded cohomology of $\Gamma$ is isometrically isomorphic to the cohomology of the complex
		\begin{equation}\label{resolution2}
			0\rightarrow L^\infty_{\rm alt}(\partial \widetilde{S})^\Gamma\rightarrow  L^\infty_{\rm alt}((\partial \widetilde{S})^2)^\Gamma\rightarrow L^\infty_{\rm alt}((\partial \widetilde{S})^3)^\Gamma\rightarrow\cdots 
		\end{equation}
	\end{thm}

	The elements of $L^\infty_{\rm alt}\left((\partial\widetilde{S})^{n+1} \right)$ are not functions, but \emph{classes} of functions.
	This being said, as it is customary we will usually denote the class of an element in $ L^\infty_{\rm alt}\left((\partial\widetilde{S})^{n+1} \right)^\Gamma$ simply by one of its representatives, thus writing  $ c\in L^\infty_{\rm alt}\left((\partial\widetilde{S})^{n+1} \right)^\Gamma$
	also when $c$ is a function.
	Notice that such a $c$ is not necessarily $\Gamma$-invariant as a function, even if its class is, and if $c$ is a cocycle, then the cocycle condition $\delta c = 0$ only holds almost everywhere in general. For later purposes, we need to single out the cases in which we can work with functions rather than with classes of functions. To this aim, we give the following:

	\begin{defn}
		We denote by $\mathcal{L}^\infty_{\text{alt}}((\partial\widetilde{S})^{\bullet+1})^\Gamma$  
		the set of bounded measurable alternating functions 
		$(\partial\widetilde{S})^{\bullet+1}\to\mathbb{R}$ which are genuinely (i.e.~not only almost everywhere) $\Gamma$-invariant. 
		
		A \emph{strict} cocycle is an element $c\in \mathcal{L}^\infty_{\text{alt}}((\partial\widetilde{S})^{\bullet+1})^\Gamma$ such that
		the identity $\delta c=0$ holds everywhere on $(\partial\widetilde{S})^{\bullet+2}$ (in this case we write $c\in Z\mathcal{L}^\infty_{\text{alt}}((\partial\widetilde{S})^{\bullet+1})^\Gamma$). 
		Any strict cocycle $c\in Z\mathcal{L}^\infty_{\text{alt}}((\partial\widetilde{S})^{n+1})^\Gamma$ defines a
		cocycle in $L^\infty_{\rm alt}((\partial\widetilde{S})^{n+1})^\Gamma$, hence, by Theorem~\ref{amenable_resolution}, it represents an element of $H_b^n(\Gamma)$.
	\end{defn}
	
	Let us now focus on the situation we are interested in. Recall that
	$\calT$ denotes the set of elements $(\xi_0,\xi_1,\xi_2)\in (\partial\widetilde{S})^3$ such that $\xi_i\neq \xi_j$ for $i\neq j$. 
	We consider the 
	map $\widehat{\omega}_g \colon (\partial\widetilde{S})^3\to \mathbb{R}$
	given by
	$$
	\widehat{\omega}_g (\xi_0,\xi_1,\xi_2)=\left\{\begin{array}{ll} \int_{T(\xi_0,\xi_1,\xi_2)} dA_{\widetilde{g}} \qquad &\text{if}\ (\xi_0,\xi_1,\xi_2)\in\mathcal{T}\ ,\\
		0 \qquad & \text{otherwise}\end{array}\right.
	$$
	(hence
	$\widehat{\omega}_g(\gamma_0,\gamma_1,\gamma_2)$ is the signed area of the ideal triangle $T(\gamma_0 {\xi}, \gamma_1 {\xi}, \gamma_2 {\xi})$,
	where the sign depends on the cyclic ordering of $\gamma_0 {\xi}, \gamma_1 {\xi}, \gamma_2 {\xi}$ on $\partial \widetilde{S}$).

	\begin{prop}\label{properties:c}
		We have $\widehat{\omega}_g\in Z\mathcal{L}^\infty_{\text{alt}}((\partial\widetilde{S})^{3})^\Gamma$, i.e.~$\widehat{\omega}_g$ is a strict cocycle. Moreover, its restriction to $\mathcal{T}$ is continuous.
	\end{prop}
	\begin{proof}
		Since the negatively curved metric $g$ is pinched, there exists $a>0$ such that 
		$k\leq -a^2$ at every point of $S$ (hence, of $\widetilde{S}$), where $k$ denotes the curvature. Therefore,
		the area of any ideal triangle of $\widetilde{S}$ is not bigger than $\pi/a^2$, hence 
		$\|\widehat{\omega}_g\|_\infty\leq \pi/a^2$, and  $\widehat{\omega}_g$ is well defined (as a set-theoretic function) and bounded; moreover,
		since elements of $\Gamma$ are orientation-preserving isometries of $\widetilde{S}$, the cochain $\widehat{\omega}_g$ is $\Gamma$-invariant, and it is alternating by construction.  The cocycle identity
		$\delta \widehat{\omega}_g=0$  follows from Stokes' Theorem, hence
		we are left to show that  $\widehat{\omega}_g$ is continuous on $\mathcal{T}$ (which implies in turn that $\widehat{\omega}_g$ is  measurable
		on $(\partial\widetilde{S})^3$). However, the fact that $\widehat{\omega}_g$ is continuous on $\mathcal{T}$  readily follows from the fact that the area of ideal triangles with distinct vertices continuously depends on the vertices. 
\end{proof}

\begin{prop}\label{rep:boundary}
	The strict cocycle $\widehat{\omega}_g$ represents the class $\theta([\omega_g])\in H^2_b(\Gamma)$.
\end{prop}

\begin{proof}
	Thanks to the fact that $\widehat{\omega}_g$ is a strict cocycle, i.e.~that it is defined everywhere and is such that
	$\delta \widehat{\omega}_g=0$ everywhere, we may apply~\cite[Corollary 2.3]{BurgerIozzi} to the case when  $G=\Gamma$,  $X=\partial\widetilde{S}$, and $Z=\Gamma\cdot \xi\subseteq X$, where $\xi\in\partial\widetilde{S}$ is a fixed basepoint, thus obtaining
	that the cochain $\widehat{\omega}_g$ represents the class $[\omega_{g,\xi}]\in H^2_b(\Gamma)$, where $\omega_{g,\xi}\in C^2_b(\Gamma)^\Gamma$ is the cocycle described in Subsection~\ref{group:sub}.
	The conclusion then follows from Lemma~\ref{boundaryrep:lemma}.
\end{proof}

\subsection{The dynamics of $\Gamma$ on $\partial \widetilde{S}$}\label{limitset}
Let $S=\widetilde{S}/\Gamma$ be as above, so that $\Gamma$ is a discrete torsion-free subgroup of $\isom^+(\widetilde{S})$. 
We will need to exploit some properties of the action of $\Gamma$ on $\partial \widetilde{S}$, which we will now briefly describe.
We first recall the definition of 
 \emph{limit set}
$\Lambda(\Gamma)$ of $\Gamma$: if $x_0\in\widetilde{S}$ is a fixed basepoint, then
$$
\Lambda(\Gamma)=\{p\in \partial \widetilde{S}\, |\, p\in \partial \widetilde{S}\cap \overline{\Gamma\cdot x_0}\}\ ,
$$
where $\overline{\Gamma\cdot x_0}$ denotes the closure  in $\widetilde{S}\cup \partial \widetilde{S}$ of the orbit of $x_0$. It is well known that this definition does not depend on the choice of $x_0$. Moreover, one says that $\Gamma$ is \emph{elementary} if $\Lambda(\Gamma)$ is finite. It is well known that for an elementary torsion-free discrete group $\Gamma<\isom^+(\widetilde{S})$ one of the following conditions must hold:
\begin{enumerate}
\item $\Gamma=\{1\}$ (and $\Lambda(\Gamma)=\emptyset$); in this case, $S=\widetilde{S}$ is diffeomorphic to $\mathbb{R}^2$;
\item $\Gamma$ is infinite cyclic generated by a parabolic element (and $\Lambda(\Gamma)$ consists of one point); in this case, $S=\widetilde{S}/\Gamma$ is diffeomorphic to a cylinder; 
\item $\Gamma$ is infinite cyclic generated by a hyperbolic element (and $\Lambda(\Gamma)$ consists of two points); in this case, $S=\widetilde{S}/\Gamma$ is again diffeomorphic to a cylinder.
\end{enumerate}

If $g\in\Gamma$ is a hyperbolic isometry, we denote by $g^+\in\partial\widetilde{S}$ (resp.~$g^-\in\partial\widetilde{S}$) the attractive (resp.~repulsive) fixed point of $g$.
We then denote by $H(\Gamma)\subseteq \partial\widetilde{S}$ the set of fixed points of hyperbolic isometries of $\Gamma$, and
we set 
$$A(\Gamma)=\{(p,q)\in(\partial \widetilde{S})^2\, |\, (p,q)=(g^-,g^+)\ \text{for some hyperbolic}\ g\in\Gamma\}$$
(here ``$A$'' stands for ``axes'',  since pairs in $A(\Gamma)$  are exactly the pairs of endpoints of the axes of  hyperbolic isometries in $\Gamma$). 
We observe that
 $$ A(\Gamma)\subseteq (H(\Gamma)\times H(\Gamma))\setminus \Delta\subseteq  (\Lambda(\Gamma)\times \Lambda(\Gamma))\setminus \Delta\ ,$$
 where $\Delta$ is the diagonal.

The following result is proved in~\cite[Theorem 1.1]{hamenstadt}:
\begin{lemma}\label{dense}
If $\Gamma$ is non-elementary, the subset $A(\Gamma)$ is dense in $(\Lambda(\Gamma)\times \Lambda(\Gamma))\setminus \Delta$.
\end{lemma}

We will need also the following:

\begin{lemma}\label{interval}
Let $\Gamma$ be non-elementary, let $I$ be a connected component of $\partial \widetilde{S}\setminus \Lambda(\Gamma)$, and let $\{x_1,x_2\}= \overline{I}\setminus I$ be the endpoints of 
$I$. Suppose that there exists an element $g\in \Gamma\setminus \{1\}$ such that $g(x_1)=x_1$. Then $g(x_2)=x_2$, and
$(x_1,x_2)\in A(\Gamma)$.
\end{lemma}
\begin{proof}
Since $g$ acts as self-homeomorphism of  $\widetilde{S}$ leaving $\Lambda(\Gamma)$ invariant, 
and $g(x_1)=x_1$, the set $g(I)$ must be a component of $\widetilde{S}\setminus \Lambda(\Gamma)$ having 
$x_1$ as an endpoint. Using that $g$ preserves the orientation of $\partial \widetilde{S}$, we then have
$g(I)=I$, 
and $g(x_2)=x_2$. Since $g\neq 1$, this implies
that $g$ is hyperbolic, and $(x_1,x_2)\in A(\Gamma)$.
\end{proof}

\section{The bounded area class is non-trivial}\label{nontrivial:sec}
Let as above $S=\widetilde{S}/\Gamma$ be a surface endowed with a complete pinched negatively curved Riemannian structure $g$.
We are now ready to prove Theorem~\ref{main:surface:thm}, which states that the class $[\omega_g]\in H^2_b(S)$ is non-trivial,
provided that $S$ is not homeomorphic to a disc or to a cylinder. 
To this aim, we first exploit a result from~\cite{BurgerIozzi}, which allows to check whether  $[\omega_g]\in H^2_b(S)$ is trivial
just by looking at the restriction of $\widehat{\omega}_g$ on triples of points contained in the limit set $\Lambda(\Gamma)$:

 \begin{prop}\label{key:prop}
 The bounded cocycle $\widehat{\omega}_g$ represents the trivial class in $H^2_b(\Gamma)$ if and only if it vanishes identically
 on $\Lambda(\Gamma)^3\subseteq (\partial \widetilde{S})^3$.
 \end{prop}
\begin{proof}
Let us first introduce some  notation from~\cite{BurgerIozzi}.
Let $X$ be a proper CAT$(-1)$-space, $G_2$ a closed subgroup of $\isom(X)$, $E$ a separable 
coefficient $G_2$-module (i.e.~the topological dual of a Banach $G_2$-module), and
$c\colon \partial X^3\to E$ 
a Borel measurable, alternating, bounded, $G_2$-invariant strict cocycle which is continuous
on the subset $\mathcal{T}$ of pairwise distinct triples in $(\partial X)^3$. 
Also denote by $[c]\in H^2_b(G_2)$ the class represented by $c$ via the Burger-Monod isomorphism
between the bounded cohomology of $G_2$ and the cohomology of the resolution
$L^\infty_{\rm alt}((\partial X)^{\bullet +1})^{G_2}$ (see Theorem~\ref{amenable_resolution}). 
Finally, let
$\pi\colon G_1\to G_2$ be 
a homomorphism, where $G_1$ is discrete, and denote by $\Lambda(\pi(G_1))\subseteq \partial X$ the limit
set of $\pi(G_1)$.

It is then proved in~\cite[Proposition 3.1]{BurgerIozzi} that if the pull-back $\pi^*([c])\in H^2_b(G_1)$ 
vanishes, then the restriction of $c$ to $\Lambda(\pi(G_1))^3\subseteq (\partial X)^3$ is identically zero. 
The conclusion follows by applying
 this result to the case when $X=\widetilde{S}$, $E=\mathbb{R}$ on which $G_2$ acts trivially, $G_1=G_2=\Gamma$, $\pi\colon \Gamma\to\Gamma$ is the identity, and $c=\widehat{\omega}_g$.
 \end{proof}

 It is now easy to deduce Theorem~\ref{main:surface:thm} from this proposition. Indeed, suppose that $[\omega_g]=0$ in $H^2_b(S)$. 
 Since $\theta\colon H^2_b(S)\to H^2_b(\Gamma)$ is an isomorphism, Proposition~\ref{rep:boundary} then implies that
 $\widehat{\omega}_g$ represents the trivial class in $H^2_b(\Gamma)$, hence $\widehat{\omega}_g$  vanishes identically
 on $\Lambda(\Gamma)^3\subseteq (\partial \widetilde{S})^3$ by Proposition~\ref{key:prop}. However, since ideal triangles in $\widetilde{S}$ have positive area,
 the cocycle $\widehat{\omega}_g$ is non-zero on any triple of distinct points in $\mathcal{T}$. Therefore, if 
  $[\omega_g]=0$ in $H^2_b(S)$
the limit set  $\Lambda(\Gamma)$ contains at most two points. 
As discussed in Subsection~\ref{limitset}, this implies that $S$ is diffeomorphic either to $\mathbb{R}^2$ or to $S^1\times\mathbb{R}$. 

Conversely, if $S$ is diffeomorphic to $\mathbb{R}^2$ or to $S^1\times \mathbb{R}$, then $\Gamma$ is amenable, hence $H^2_b(S)\cong H^2_b(\Gamma)=0$, hence $[\omega_g]=0$ in $H^2_b(S)$. This concludes the proof of Theorem~\ref{main:surface:thm}.

\begin{rem}
Theorem~\ref{main:surface:thm} admits a different, more elementary proof, which does not make use of amenable actions. In fact, if $S$ is a complete pinched negatively curved
surface which is not diffeomorphic to $\mathbb{R}^2$ or to $S^1\times \mathbb{R}$, then by~\cite[Corollary 4.13]{portilla} $S$ contains a \emph{generalized Y-shaped piece} $S'$, i.e.~a 
convex subsurface $S'$  with (possibly empty) geodesic boundary homeomorphic to the sphere with $m$ punctures (corresponding to cusps of finite area) 
and $n$ open discs removed,
with $n+m=3$. It may also happen that two geodesic boundary components of such a generalized $Y$-shaped piece are glued to each other in $S$; in this case, 
the generalized $Y$-piece gives rise to either a once-punctured torus or a one-holed torus with geodesic boundary.
In any case, any generalized $Y$-piece in the decomposition of $S$ provides a convex subsurface $S'$ of finite area with compact geodesic boundary. By exploiting 
the fact that the boundary components and/or the cusps of $S'$ have an amenable fundamental group, together with the fact that
the simplicial volume of $S'$ does not vanish, it is now easy to prove that the pull-back of the class $[\omega_g]\in H^2_b(S)$ in $H^2_b(S')$ is non-trivial. This implies
in particular that $[\omega_g]\neq 0$ in $H^2_b(S)$.

However, the framework introduced above via amenable actions will prove useful in the following sections, where we will investigate  the dependence of the bounded area class on the metric.
\end{rem}

\section{Semi-conjugacy of circle actions}\label{semi-conj:sec}
In this section we first recall the definition of semi-conjugacy, and then we focus on the proof of Theorem~\ref{semi-conj:thm}, which states that the circle actions associated to any pair of pinched negatively curved structures on the same surface $S$ are semi-conjugate to each other.

Let $\operatorname{Homeo}^+(S^1)$ be the group of orientation preserving homeomorphisms of the circle, let $\Gamma$ be a group, and let us consider
an action of $\Gamma$ on $S^1$, i.e.~a representation $\rho\colon \Gamma\to \operatorname{Homeo}^+(S^1)$. To such an action there is associated a very interesting bounded cohomology class, the \emph{integral bounded Euler class}, 
which carries a lot of information about the dynamics of the action.
Indeed, a fundamental result by Ghys ensures that two circle actions of the same group share the same integral bounded Euler class if and only if they are semi-conjugate~\cite{Ghys0,Ghys1,Ghys2}. 
For a detailed account on the history of the notion of semi-conjugacy, together with a discussion of the equivalence of several definitions we refer the reader to~\cite{BFH}.

	Let us fix an identification between $S^1$ and the quotient $\mathbb{R}/\mathbb{Z}$, with corresponding  covering projection $\mathbb{R}\to S^1$.
	Let us consider an ordered $k$-tuple $(x_1,\ldots,x_k)\in (S^1)^k$. We say that such a $k$-tuple is
		\begin{itemize}
			\item \emph{positively oriented} if there exist lifts $\widetilde{x}_i\in \mathbb{R}$ of the $x_i$ such that
			$$
			\widetilde{x}_1<\widetilde{x}_2<\dots<\widetilde{x}_k<\widetilde{x}_1+1\ ,
			$$
			\item 
			\emph{weakly positively oriented} if there exist
			lifts $\widetilde{x}_i\in \mathbb{R}$ of the $x_i$ such that
			$$
			\widetilde{x}_1\leq \widetilde{x}_2\leq \dots\leq \widetilde{x}_k\leq \widetilde{x}_1+1\ .
			$$

			\end{itemize}
		Then, we say that
		a (not necessarily continuous) map $\varphi\colon S^1\to S^1$ is \emph{increasing of degree one} if every positively oriented $k-$tuple is sent to a weakly positively oriented $k$-tuple (for example, any constant map $\varphi\colon S^1\to S^1$ is increasing of degree one). It turns out  that it is sufficient to check this property on $4$-tuples, i.e., a map is increasing of degree one if the following condition holds:
		if $(x_1,\ldots,x_4)\in (S^1)^4$ is a positively oriented 4-tuple, then 
		$(\varphi(x_1),\dots,\varphi(x_4))$ is weakly positively oriented (see~\cite[Lemma 10.12]{Frigeriobook}).

		\begin{defn}
			Let $\rho_1,\rho_2\colon G\to \omeo$ be circle actions.
			A \emph{left semi-conjugacy} from $\rho_1$ to $\rho_2$ is a non-decreasing degree one map $\varphi\colon S^1 \to S^1$ such that
			\[\rho_1(\gamma)\varphi=\varphi\rho_2(\gamma) \quad \text{for every} \quad \gamma\in G.\]
			If such a map exists, the action $\rho_1$ is said to be \emph{left semi-conjugate} to $\rho_2$ and $\rho_2$ is \emph{right semi-conjugate} to $\rho_1$. Finally, we say that $\rho_1$ and $\rho_2$
			are \emph{semi-conjugate} if they are both left semi-conjugate and right semi-conjugate.
		\end{defn}
	
	A complete pinched negatively curved metric on a surface naturally induces an action of its fundamental group on the circle. Indeed, let $S$ be an oriented surface (possibly, of infinite type) endowed with such a metric $g$ and let $\Gamma\coloneqq \pi_1(S)$ be its fundamental group. 
The universal covering $\widetilde{S}$ is diffeomorphic to an open disc and has a natural boundary $\partial \widetilde{S}$ homeomorphic to $S^1$. We may thus fix an identification of $\partial \widetilde{S}$ with $S^1$. The metric $g$ induces a holonomy representation $\Gamma\to \isom^+(\widetilde{S})$, which is well defined up to conjugacy. An element of $\isom^+(\widetilde{S})$ uniquely extends to an element of $\operatorname{Homeo}^+(\partial\widetilde{S})\cong \omeo$. Thus, we get a circle action which we will denote by $\rho_g\colon \Gamma\to \omeo$. Different holonomy representations of the same Riemannian structure and different identifications between $\partial \widetilde{S}$ and $S^1$ induce conjugate (hence, semi-conjugate) circle actions.

\subsection{Proof of Theorem~\ref{semi-conj:thm}}

The rest of the section will be devoted to the proof of Theorem~\ref{semi-conj:thm}, which will be broken down in some preliminary lemmas and propositions.
Let $g_1$, $g_2$ be complete pinched negatively curved structures on $S$, and let $\rho_1,\rho_2\colon \pi_1(S)\to \omeo$ be the corresponding
circle actions.

If $\pi_1(S)$ is  trivial or isomorphic to $\mathbb{Z}$, then both $\rho_1$ and $\rho_2$ admit a global fixed point, hence they are semi-conjugate (see, e.g., \cite[Corollary 4.2]{BFH}). We may thus assume that $\rho_i(\pi_1(S))$ is non-elementary for $i=1,2$.

Let $p_i\colon \widetilde{S}_i\to S$ be the Riemannian universal covering of $S$ associated to the structure $g_i$, $i=1,2$. Also fix a lift of the identity, i.e., an orientation-preserving homeomorphism
$\psi\colon \widetilde{S}_1\to \widetilde{S}_2$ such that $p_2\circ \psi=p_1$. With a slight abuse, we denote by $\rho_i\colon \pi_1(S)\to \isom^+(\widetilde{S}_i)$ both the 
holonomy representation associated to the covering $p_i$, and 
the induced circle action $\rho_i\colon \pi_1(S)\to \text{Homeo}^+(\partial \widetilde{S}_i)$.
By construction, for every $\gamma\in \pi_1(S)$
we have $\psi\circ \rho_1(\gamma)=\rho_2(\gamma)\circ \psi$ as maps from $\widetilde{S}_1$ to $\widetilde{S}_2$.

We would like to define an increasing map of degree one $\varphi\colon \partial \widetilde{S}_1\to \partial \widetilde{S}_2$ by ``extending'' $\psi$ to the boundary at infinity. Unfortunately,
no continuous extension of $\psi$ exists in general (even when $S$ is of finite type), hence some work is in order. 

With notation as in Subsection~\ref{limitset}, for $i=1,2$, let us set
$H_i=H(\rho_i(\pi_1(S)))$, $A_i=A(\rho_i(\pi_1(S))$ and $\Lambda_i=\Lambda(\rho_i(\pi_1(S))$.
Also recall that, if  $g\in\rho_i(\pi_1(S))$ is a hyperbolic isometry, we denote by $g^+\in\partial\widetilde{S}_i$ (resp.~$g^-\in\partial\widetilde{S}_i$) the attractive (resp.~repulsive) fixed point of $g$.
Moreover,  if $g$ is parabolic, then we set $g^+=g^-=p$, where $p$ is the unique fixed point of $g$ in $\partial\widetilde{S}_i$.

We first define 
the map $\varphi$ on the subset $H_1\subseteq \partial\widetilde{S}_1$ as follows: if $p\in H_1$, then there exists a unique primitive element $\gamma\in \pi_1(S)$ such that $p=\rho_1(\gamma)^+$,
and we set $\varphi(p)=\rho_2(\gamma)^+$ (note that, in this way, if $\rho_2(\gamma)$ is parabolic then both the attractive
and the repulsive fixed points of $\rho_1(\gamma)$  are sent to the same point in $\partial \widetilde{S}_2$, hence $\varphi$ is not injective).

\begin{lemma}\label{order1}
Let $(x_1,\ldots,x_4)\in (\partial \widetilde{S}_1)^4$ be a positively oriented 4-tuple such that $(x_1,x_3)\in A_1$ and $(x_2,x_4)\in A_1$. 
Then  $(\varphi(x_1),\varphi(x_3))\in A_2$, $(\varphi(x_2),\varphi(x_4))\in A_2$ and
$(\varphi(x_1),\ldots,\varphi(x_4))\in (\partial \widetilde{S}_2)^4$ is positively oriented.
\end{lemma}
\begin{proof}
The hypothesis implies 
that  there exist hyperbolic isometries $\gamma,\gamma'\in \pi_1(S)$ such that 
$x_1=\rho_1(\gamma)^-$, $x_3=\rho_1(\gamma)^+$, $x_2=\rho_1(\gamma')^-$, $x_4=\rho_1(\gamma')^+$.
We denote by $\alpha\subseteq \widetilde{S}_1$ (resp.~ $\alpha'\subseteq \widetilde{S}_1$)  the axis of the hyperbolic isometry of $\rho_1(\gamma)$
(resp.~of $\rho_1(\gamma')$), oriented from $x_1$ to $x_3$
(resp.~from $x_2$ to $x_4$).

We first observe that both $\rho_2(\gamma)$ and $\rho_2(\gamma')$ are hyperbolic. In fact, if one of them, say $\rho_2(\gamma)$, were parabolic,
then both endpoints of $\psi(\alpha)$ would coincide with a single point $\overline{y}\in\partial \widetilde{S}_2$. Moreover,
the closure at infinity of one of the two connected components of $\widetilde{S}_2\setminus \psi(\alpha)$ would contain only the point $\overline{y}$. Since such component
would contain a half-line of $\psi(\alpha')$, also the isometry $\rho_2(\gamma')$ would fix $\overline{y}$. Therefore, the isometries
$\rho_2(\gamma)$ and $\rho_2(\gamma')$ would share a common fixed point. Since they generate a discrete subgroup of $\isom^+(\widetilde{S}_2)$, this would imply
that $\rho_2(\gamma)^m=\rho_2(\gamma')^n$ for some $m\neq 0$, $n\neq 0$, hence $\gamma^m=(\gamma')^n$. But this contradicts the fact that
the fixed points of $\rho_1(\gamma)$ are disjoint from the fixed points of $\rho_1(\gamma')$. 

We then denote by $y_1,y_3$ the repulsive and the attracting fixed points of $\rho_2(\gamma)$ and 
by $y_2,y_4$ the repulsive and the attracting fixed points of $\rho_2(\gamma')$. Since $\psi\circ \rho_1(\gamma')=\rho_2(\gamma')\circ\psi$, we have
$$y_1=\lim_{t\to -\infty} \psi(\alpha(t))\ ,\qquad  y_3=\lim_{t\to +\infty} \psi(\alpha(t))\ ,$$
$$y_2=\lim_{t\to -\infty} \psi(\alpha'(t))\ , \qquad y_4=\lim_{t\to +\infty} \psi(\alpha'(t))\ .$$
Since $(x_1,x_2,x_3,x_4)$ is positively oriented,  the axes
$\alpha$ and $\alpha'$ transversely intersect at one point $p\in\widetilde{S}_1$,  and the speeds
of $\alpha$ and $\alpha'$ at $p$ define a positively oriented basis of $T_p \widetilde{S}_1$. Since $\psi$ is orientation-preserving,
this implies that also the topological lines
$\psi(\alpha)$ and $\psi(\alpha')$ transversely intersect at $\psi(p)\in\widetilde{S}_2$, and
the speeds
of $\psi\circ\alpha$ and $\psi\circ\alpha'$ at $\psi(p)$ define a positively oriented basis of $T_{\psi(p)} \widetilde{S}_2$. These facts readily imply that the 
$4$-tuple $(\varphi(x_1),\ldots,\varphi(x_4))=(y_1,y_2,y_3,y_4)$ is positively oriented, thus concluding the proof.
\end{proof}

\begin{lemma}\label{order2}
Let $n\in \matN$ and let $(y_1,z_1,y_2,z_2,\ldots,y_n,z_n)\in (\partial \widetilde{S}_1)^{2n}$ be a positively oriented $2n$-tuple such that $(y_i,z_i)\in A_1$ for every $i=1,\dots,n$. 
Then $$(\varphi(y_1),\varphi(z_1)\ldots,\varphi(y_{n}),\varphi(z_n))\in (\partial \widetilde{S}_2)^{2n}$$ is weakly positively oriented.
\end{lemma}

\begin{proof}
For each $i = 1, \dots, n$, let $\gamma_i \in \pi_1(S)$ be the primitive element such that $y_i = \rho_1(\gamma_i)^-$ and $z_i = \rho_1(\gamma_i)^+$. We denote by
$\alpha_i \subseteq \widetilde{S}_1$ the axis of  $\rho_1(\gamma_i)$, oriented from $y_i$ to $z_i$.

Because the $2n$-tuple $(y_1, z_1, \dots, y_n, z_n)$ is positively oriented, the axes $\alpha_i$ are pairwise disjoint. Let $C \subseteq \widetilde{S}_1$ be the closure in $\partial\widetilde{S}\cup \widetilde{S}$ of the region bounded by $\bigcup_{i=1}^n \alpha_i$. By construction, the orientation of the boundary $\partial C$ agrees with the orientation of each individual axis $\alpha_i$.

The map $\psi: \widetilde{S}_1 \to \widetilde{S}_2$ is an orientation-preserving diffeomorphism. Therefore, it sends the oriented boundary of $C$ into the oriented boundary of $\psi(C)$. Since $\psi \circ \rho_1(\gamma_i) = \rho_2(\gamma_i) \circ \psi$, we have
$$ \lim_{t \to -\infty} \psi(\alpha_i(t))=\rho_2(\gamma_i)^-=\varphi(y_i)  \quad \text{and} \quad \lim_{t \to +\infty} \psi(\alpha_i(t))=\rho_2(\gamma_i)^+=\varphi(z_i). $$
Note that, if $\rho_2(\gamma_i)$ is parabolic, then $\rho_2(\gamma_i)^-=\rho_2(\gamma_i)^+$, hence  $\varphi(y_i) = \varphi(z_i)$. 

Because $\psi$ sends the oriented boundary of $C$ onto the oriented boundary of $\psi(C)$, the $2n$ limit points obtained as the closure of the topological lines $\psi(\alpha_i)$ must appear in $\partial\widetilde{S}_2$ in the same  cyclic order as the endpoints of the $\alpha_i$ in $\partial \widetilde{S}_1$. Since $(y_1,z_1,y_2,z_2,\ldots,y_n,z_n)\in (\partial \widetilde{S}_1)^{2n}$ is positively oriented,
and considering the 
possible collapse of the points $\varphi(y_i)$ and $\varphi(z_i)$ when $\rho_2(\gamma_i)$ is parabolic, 
we may conclude that the resulting $2n$-tuple $(\varphi(y_1), \varphi(z_1), \dots, \varphi(y_n), \varphi(z_n))$ is weakly positively oriented.
\end{proof}

We are now ready to prove that $\varphi$ is order-preserving on $4$-tuples of  fixed points of hyperbolic isometries in $\rho_1(\pi_1(S))$:

\begin{prop}\label{order3}
Let  $(x_1,\ldots,x_4)\in H_1^4$ be a positively oriented 4-tuple. Then 
 $(\varphi(x_1),\dots,\varphi(x_4))$ is weakly positively oriented. 
 \end{prop}
\begin{proof}
For every $i=1,2$ and $x,y\in\partial \widetilde{S}_i$ with $x\neq y$, we denote by $\wideparen{xy}$ the  open  arc in $\partial \widetilde{S}$ with starting point $x$ and ending point $y$, where we understand that arcs are parametrized according to the orientation of $\partial \widetilde{S}_i$.
We choose pairwise disjoint neighbourhoods $U_i$ of the $x_i$ in $\partial \widetilde{S}_1$, in such a way that $U_i$ does not contain the
fixed point $x_i'\neq x_i$ of the hyperbolic isometries
in $\rho_1(\pi_1(S))$ fixing $x_i$. 

For every $i=1,2,3,4$ we replace $x_i$ with a pair $(y_i,z_i)\in A_1$ as follows. 

We first consider the generic case when $x_i$ is \emph{not} the endpoint of any interval in $\partial \widetilde{S}_1\setminus \Lambda_1$. In this case, 
each component of $U_i\setminus \{x_i\}$ intersects the limit set $\Lambda_1$, hence Lemma~\ref{dense} implies that there exists a pair
$(y_i,z_i)\in A_1$ such that $y_i\in U_i$, $z_i\in U_i$ and the triple $(y_i,x_i,z_i)$ is positively oriented. Moreover,
our choice for $U_i$ and Lemma~\ref{order1} (applied to the $4$-tuple $(x_i',y_i,x_i,z_i)$) imply that
 $(\varphi(y_i),\varphi(x_i),\varphi(z_i))$ is positively oriented, i.e., that $\varphi(x_i)\in\wideparen{\varphi(y_i)\varphi(z_i)}$.

If $x_i$ is the endpoint of an open interval $I$ in $\partial \widetilde{S}_1\setminus \Lambda_1$, then we choose $y_i$ and $z_i$ in such a way 
that $I=\wideparen{y_iz_i}$ (so that either $x_i=y_i$ or $x_i=z_i$). Thanks to Lemma~\ref{interval}, we have indeed $(y_i,z_i)\in A_1$, and by construction
the triple $(y_i,x_i,z_i)$ is weakly positively oriented. 

Let us now suppose that no pair $x_i,x_j$, $i\neq j$, is the pair of endpoints of a component of $\partial \widetilde{S}_1\setminus \Lambda_1$.
Then, it readily follows from our construction that $(y_1,z_1,y_2,z_2,y_3,z_3,y_4,z_4)\in (\partial \widetilde{S}_1)^8$ is positively oriented. By Lemma~\ref{order2},
this implies in turn that $(\varphi(y_1),\varphi(z_1)\ldots,\varphi(y_{4}),\varphi(z_4))\in (\partial \widetilde{S}_2)^{8}$ is weakly positively oriented. Recall
now, that, for every $i=1,\dots,4$, we have either $\varphi(x_i)\in\wideparen{\varphi(y_i)\varphi(z_i)}$, or $\varphi(x_i)=\varphi(y_i)$, or $\varphi(z_i)$. Thus
the $12$-tuple $$ (\varphi(y_1),\varphi(x_1), \varphi(z_1)\ldots,\varphi(y_{4}),\varphi(x_4),\varphi(z_4))\in (\partial \widetilde{S}_2)^{12}$$ is weakly positively oriented,
hence, \emph{a fortiori}, the $4$-tuple  $(\varphi(x_1),\dots,\varphi(x_4))\in (\partial \widetilde{S}_2)^{4}$ is weakly positively oriented.

If there exist $x_i,x_j$, $i\neq j$, which are the pair of endpoints of a component of $\partial \widetilde{S}_1\setminus \Lambda_1$, then the configuration we need to investigate is even easier. Of course, $x_i$ and $x_j$ must be consecutive, hence without loss of generality we may assume $i=1$ and $j=2$. If also $x_3,x_4$ are the
endpoints of a component of $\partial \widetilde{S}_1\setminus \Lambda_1$, then Lemma~\ref{interval} ensures that $(x_1,x_2)\in A_1$ and $(x_3,x_4)\in A_1$, and the conclusion
follows from Lemma~\ref{order2}. Otherwise, let $y_3,z_3,y_4,z_4$ be defined as above. Arguing as before, we get that
$(x_1,x_2,y_3,z_3,y_4,z_4)$ is weakly positively oriented, hence also $(\varphi(x_1),\varphi(x_2),\varphi(y_3),\varphi(z_3),\varphi(y_4),\varphi(z_4))$
is weakly positively oriented by Lemma~\ref{order2}. The conclusion follows as above from the fact that, for $i=3,4$,   we have either $\varphi(x_i)\in\wideparen{\varphi(y_i)\varphi(z_i)}$, or $\varphi(x_i)=\varphi(y_i)$, or $\varphi(z_i)$.
\end{proof}

\begin{cor}\label{order4}
Let  $(x_1,x_2,x_3)\in H_1^3$ be a positively oriented triple. Then the triple
 $(\varphi(x_1),\varphi(x_2),\varphi(x_3))$ is weakly positively oriented. 
 \end{cor}
 \begin{proof}
 Since $\rho_1(\pi_1(S))$ is not elementary, we may extend any positively oriented triple $(x_1,x_2,x_3)\in H_1^3$ to a positively oriented $4$-tuple
  $(y_1,y_2,y_3,y_4)\in H_1^4$, where $x_1=y_{i_1}$, $x_2=y_{i_2}$, $x_3=y_{i_3}$, $i_1<i_2<i_3$.  Proposition~\ref{order3} now implies
  that $(\varphi(y_1),\dots,\varphi(y_4))$ is weakly positively oriented, hence also
  $(\varphi(y_{i_1}),\varphi(y_{i_2}),\varphi(y_{i_3}))=(\varphi(x_1),\varphi(x_2),\varphi(x_3))$ is weakly positively oriented. 
 \end{proof}

Before extending $\varphi$ to the whole $\partial \widetilde{S}_1$, we observe that $\varphi\colon H_1\to \partial \widetilde{S}_2$ is equivariant with respect to the actions $\rho_1$, $\rho_2$. In fact, if $\gamma\in \pi_1(S)$
and $x=\rho_1(\gamma_0)^+$ for some $\gamma_0\in\pi_1(S)$, then $\rho_1(\gamma)(x)$ is the attractive fixed point of
$\rho_1 (\gamma \gamma_0\gamma^{-1})$, hence 
$$\varphi(\rho_1(\gamma)(x))=\rho_2 (\gamma \gamma_0\gamma^{-1})^+=\rho_2(\gamma)(\rho_2(\gamma_0)^+)=
\rho_2(\gamma)(\varphi(x))\ .$$
We have thus shown that
\begin{equation}\label{equivariance}
\varphi(\rho_1(\gamma)(x))=\rho_2(\gamma)(\varphi(x))\quad \text{for every}\ \gamma\in\pi_1(S)\, ,\ x\in H_1 \ .
\end{equation}

Let us now choose an identification of $\partial\widetilde{S}_1$  with $S^1=\mathbb{R}/\mathbb{Z}$, in such a way that $[0]\in H_1\subseteq S^1\cong \partial S_1$.
Also identify  $\partial\widetilde{S}_2$ with $S^1$ in such a way that $\varphi([0])=[0]$. Let us denote by $p\colon \mathbb{R}\to \mathbb{R}/\mathbb{Z}=S^1$ the covering 
projection, and set $\widetilde{H}_1=p^{-1}(H_1)\subseteq \mathbb{R}$. 

\begin{prop}\label{liftH}
The map $\varphi\colon H_1\to S^1$ lifts to a map $\widetilde{\varphi}\colon \widetilde{H}_1\to \mathbb{R}$ such that the following conditions hold:
\begin{enumerate}
\item $\widetilde{\varphi}$ is weakly increasing;
\item $p\circ \widetilde{\varphi}=\varphi\circ p$;
\item $\widetilde{\varphi}(\widetilde{x}+1)=\widetilde{\varphi}(\widetilde{x})+1$ for every $\widetilde{x}\in\mathbb{R}$.
\end{enumerate}
\end{prop}
\begin{proof}
It is sufficient to define a weakly increasing map $\widetilde{\varphi}\colon \widetilde{H}_1\cap [0,1) \to [0,1]$ in such a way that
 $p( \widetilde{\varphi}(\widetilde{x}))=\varphi (p(\widetilde{x}))$ for every $\widetilde{x}\in \widetilde{H}_1\cap [0,1)$. Indeed, we may then extend such a $\widetilde{\varphi}$ to the whole $\widetilde{H}_1$  by periodicity, thus getting
 the desired lift of $\varphi$.
 
We set $\widetilde{\varphi}(0)=0$. Moreover, for every $\widetilde{x}\in \widetilde{H}_1\cap (0,1)$, we define $\widetilde{\varphi}(\widetilde{x})$ as follows:
\begin{enumerate}
\item
if $\varphi(p(\widetilde{y}))=[0]$ for every $\widetilde{y}\in \widetilde{H}_1\cap [0,\widetilde{x}]$, then $\widetilde{\varphi}(\widetilde{x})=0$; 
\item
if $\varphi(p(\widetilde{x}))=[0]$ but there exists $\widetilde{y}\in \widetilde{H}_1\cap [0,\widetilde{x})$ such that $\varphi(p(\widetilde{y}))\neq [0]$, then $\widetilde{\varphi}(\widetilde{x})=1$;
\item
if $\varphi(p(\widetilde{x}))\neq[0]$, then $\widetilde{\varphi}(\widetilde{x})$ is the unique element of $(0,1)$ such that
$p( \widetilde{\varphi}(\widetilde{x}))=\varphi (p(\widetilde{x}))$.
\end{enumerate}

	By construction, $\widetilde{\varphi}$ takes values in $[0,1]$ and is such that  $p( \widetilde{\varphi}(\widetilde{x}))=\varphi (p(\widetilde{x}))$ for every $x\in \widetilde{H}_1\cap [0,1)$.
	Therefore, in order to conclude we need to show that 
	$\widetilde{\varphi}(\widetilde{x}_1)\leq \widetilde{\varphi}(\widetilde{x}_2)$ for every $\widetilde{x}_1,\widetilde{x}_2\in\widetilde{H}_1\cap [0,1)$ such that $0\leq \widetilde{x}_1< \widetilde{x}_2<1$.
			
	Suppose first that $\varphi(p(\widetilde{x}_1))=[0]$. If $\varphi(p(\widetilde{y}))=[0]$ for every $\widetilde{y}\in \widetilde{H}_1\cap [0,x_1)$, then $\widetilde{\varphi}(\widetilde{x}_1)=0$,
	and there is nothing to prove. Otherwise, there exists $\widetilde{y}\in \widetilde{H}_1\cap [0,\widetilde{x}_1)$ such that $\varphi(p(\widetilde{y}))\neq [0]$, and $\widetilde{\varphi}(\widetilde{x}_1)=1$. Now, since the quadruple $(p(0),p(\widetilde{y}), p(\widetilde{x}_1),p(\widetilde{x}_2))$ is positively oriented, 
	by Proposition~\ref{order3} also the quadruple $$([0]=\varphi(p(0)),\varphi(p(\widetilde{y})), [0]=\varphi(p(\widetilde{x}_1)),\varphi(p(\widetilde{x}_2)))$$
	is weakly positively oriented. This forces
	$p(\widetilde{x}_2)=[0]$, hence $\widetilde{\varphi}(\widetilde{x}_2)=1$, and again $\widetilde{\varphi}(\widetilde{x}_1)\leq \widetilde{\varphi}(\widetilde{x}_2)$.
	
	Suppose now $\varphi(p(\widetilde{x}_1))\neq[0]$. If $\varphi(p(\widetilde{x}_2))=[0]$, then our construction implies that
	$\widetilde{\varphi}(\widetilde{x}_2)=1$, hence $\widetilde{\varphi}(\widetilde{x}_1)\leq \widetilde{\varphi}(\widetilde{x}_2)$.
	Otherwise, both $\widetilde{\varphi}(\widetilde{x}_1)$ and $\widetilde{\varphi}(\widetilde{x}_{2})$ belong to $(0,1)$.
	Since $([0],p(\widetilde{x}_1),p(\widetilde{x}_2))$ is positively oriented, Corollary~\ref{order4} ensures that also
	$(\varphi(p(0)), \varphi(p(\widetilde{x}_1)),\varphi(p(\widetilde{x}_2)))=([0], \varphi(p(\widetilde{x}_1)),\varphi(p(\widetilde{x}_2)))$ is weakly positively oriented.
	Due to our definition of $\widetilde{\varphi}$, these facts imply that $\widetilde{\varphi}(\widetilde{x}_1)\leq \widetilde{\varphi}(\widetilde{x}_2)$. 
	\end{proof}

It is now easy to extend $\widetilde{\varphi}$ to the whole real line, by setting, for every $t\in\mathbb{R}$,
$$
\widetilde{\varphi}(t)=\sup \{ \widetilde{\varphi}(t')\, |\,  t'\in H_1\, ,\ {t'\leq t} \}\ .
$$
It is immediate to check that
 Properties (1), (2) and (3) of Proposition~\ref{liftH} also hold for the map $\widetilde{\varphi}\colon \mathbb{R}\to \mathbb{R}$
just defined. Property (3) implies that $\widetilde{\varphi}$ induces a map $\varphi\colon S^1\to S^1$ which extends the map $\varphi$ already defined on $H_1$. 
Moreover, using that every element of $\rho_i(\pi_1(S))$, $i=1,2$, is orientation-preserving, one easily deduces from Equation~\eqref{equivariance} that
$$
\varphi(\rho_1(\gamma)(x))=\rho_2(\gamma)(\varphi(x))\quad \text{for every}\ \gamma\in\pi_1(S)\, ,\ x\in S^1 \ .
$$
Finally, being induced by a weakly increasing map $\widetilde{\varphi}$ which commutes with integral translations, the map
 $\varphi\colon S^1\to S^1$ is increasing of degree one.

 We have thus shown that $\rho_1$ is right semi-conjugate to $\rho_2$. By switching the r\^oles of $\rho_1$ and $\rho_2$ we may prove that 
 $\rho_2$ is right semi-conjugate to $\rho_1$ as well. Hence, $\rho_1$ and $\rho_2$ are semi-conjugate, and this concludes the proof of Theorem~\ref{semi-conj:thm}.

\section{Bounded Euler class vs.~bounded Area class}\label{euler:sec}
This section is devoted to the proof of Theorem~\ref{hyperbolic:equal}. Before beginning with the proof, we recall some facts about the bounded Euler class. Let  $\Gamma$ be a group acting on the circle, i.e., let us fix a representation $\rho\colon\Gamma\to \operatorname{Homeo}^+(S^1)$. 
To such a circle action there is associated a bounded cohomology class with integral coefficients,
called \emph{integral bounded Euler class} (see e.g.~\cite[Sections 10.1, 10.2]{Frigeriobook} for the definition). As mentioned in the introduction (and in the previous section), this class completely classifies the action up to semi-conjugacy, thus carrying a lot of information
on its dynamics. In this paper we are mainly interested in bounded classes with real coefficients, hence we will focus our attention on the \emph{real} Euler
bounded class, which is obtained from the integral one just by the usual change of coefficients map. We denote the real bounded Euler class
of the representation $\rho\colon \Gamma\to \operatorname{Homeo}^+(S^1)$ by $e_b(\rho)$, so that $e_b(\rho)$ is an element of $H^2_b(\Gamma)$. 
It turns out that $e_b(\rho)$ admits a very simple representative in $Z\mathcal{L}^\infty_{\text{alt}}((S^1)^{3})^\Gamma$, which we are now going to describe.

\subsection{The orientation cocycle}\label{or:sub}

		Recall that $\mathcal{T}$ denotes the set of triples $(x_0,x_1,x_2)\in (S^1)^3$ such that $x_i\neq x_j$ for $i\neq j$, and let $\ori\colon ({S^1})^3\to\mathbb{R}$  be defined as follows:
		$$
		\ori (x_0,x_1,x_2)=\left\{\begin{array}{ll} +1 & \text{if}\  (x_0,x_1,x_2)\in\mathcal{T}\ \text{is positively oriented}\\
			-1 & \text{if}\  (x_0,x_1,x_2)\in\mathcal{T}\ \text{is negatively oriented}\\
			0 & \text{if}\  (x_0,x_1,x_2)\notin\mathcal{T}\ .\end{array}\right.
		$$
		
				By construction, the map $\ori$ is bounded and measurable, i.e.,~it belongs to
		$\mathcal{L}^\infty_{\text{alt}}((S^1)^{3})$. Moreover, it is straightforward
		to check that $\delta \ori=0$, and since elements of $\Gamma$ act on $S^1$ via orientation-preserving homeomorphisms,
		the map $\ori$ is $\Gamma$-invariant. Being also alternating, $\ori$ is then
		 a strict cocycle in $Z\mathcal{L}^\infty_{\text{alt}}((S^1)^{3})^\Gamma$ (here and henceforth, we understand that $\Gamma$ acts on $S^1$, hence on $(S^1)^k$, via $\rho$). 
		 We denote by $[\ori_\rho]\in H^2_b(\Gamma)$ the bounded cohomology class represented by the strict cocycle $\ori$ (observe that, even if $\ori$ as a map does not depend on $\rho$, the bounded cohomology class it represents depends on $\rho$, as $\rho$ enters into the definition of the cochain complex 
		 $\mathcal{L}^\infty_{\text{alt}}((S^1)^{\bullet +1})^\Gamma$).
		 
		 The following result is proved in~\cite[Lemma 2.1]{Iozzi} (see also the proof of Proposition 1.7 in the same paper):
		
		\begin{prop}\label{proportionality}
			For a representation $\rho\colon \Gamma\to \omeo$, it holds that $$e_b(\rho) =-\frac{1}{2} [\ori_\rho]\qquad  \text{in}\  H_b^2(\Gamma)\ .$$
		\end{prop}
	
		Let now $S$ be an oriented surface endowed with a hyperbolic metric $g$ and let $\Gamma\coloneqq \pi_1(S)$ be its fundamental group. Let $\rho_g\colon \Gamma\to \omeo$ be the associated circle action. It is well known that all ideal triangles in the hyperbolic plane have an area equal to $\pi$, hence, if $\widehat{\omega}_g\in Z\mathcal{L}^\infty_{\text{alt}}((S^1)^{3})^\Gamma$ is the strict cocycle defined in 
		Subsection~\ref{sec: bc as meas fun}, we get $\widehat{\omega}_g=\pi \cdot \ori$. Thanks to Proposition~\ref{proportionality}, this implies in turn that
	\begin{equation}\label{prop:eq}
		\theta([\omega_g])=\pi [\ori_{\rho_g}]=-2\pi e_b(\rho_g)\ . 
	\end{equation}

Let now $g_1,g_2$ be complete hyperbolic structures on the  surface $S$, and for $i=1,2$ let $\rho_i\colon \Gamma\to \omeo$ be the associated circle actions. Theorem~\ref{semi-conj:thm} ensures that $\rho_1$ is semi-conjugate to $\rho_2$. By Ghys' result on semi-conjugacy, it follows that they induce the same \emph{integral} bounded Euler class. 
		Since the real bounded Euler class is obtained by the integral bounded Euler class via the change of coefficients map, 
		the fact that $\rho_1$ is semi-conjugate to $\rho_2$ implies then that $e_b(\rho_1)=e_b(\rho_2)$, hence, by Equation~\eqref{prop:eq}, that
		$$
		\theta([\omega_{g_1}])=\theta([\omega_{g_2}])\ .
		$$
		This concludes the proof of Theorem~\ref{hyperbolic:equal}.

\section{Distinguishing the case of constant curvature}\label{nonhyp:sec}
We have shown in the previous section that the bounded area class cannot distinguish non-isometric hyperbolic structures on the same surface. 
However, in this section we prove that, at least for compact surfaces, it recognizes hyperbolic structures among all pinched negatively curved structures.
Let $S=\widetilde{S}/\Gamma$ be a closed surface endowed with a negatively curved metric $g$. Let $\rho_g\colon \pi_1(S)\to \textrm{Homeo}^+(\partial \widetilde{S})$ be 
the circle action associated to $g$, and recall that $e_b(\rho_g)\in H^2_b(\pi_1(S))$ is the real bounded Euler class of $\rho_g$.

The following result readily implies Theorem~\ref{recognizing:hyperbolic} from the introduction:

\begin{thm}\label{eulerhyp}
Let $S$ be a closed surface with negative Euler characteristic. We have
$$
\theta([\omega_g])=-2\pi e_b(\rho_g)
$$
if and only if $S$ is hyperbolic. 
Therefore, if $g_1$ and $g_2$ are   negatively curved Riemannian metrics on $S$ such that 
$[\omega_{g_1}]=[\omega_{g_2}]$ in $H^2_b(S)$, then $g_1$ is hyperbolic if and only if $g_2$ is hyperbolic.
\end{thm}
\begin{proof}
We know from the previous section that,
if $g$ is hyperbolic, then $\theta([\omega_g])=\pi [\ori_{\rho_g}]=-2\pi e_b(\rho_g)$.

Conversely, suppose $\theta([\omega_g])=-2\pi e_b(\rho_g)$. By Proposition~\ref{proportionality} we deduce that
 both $\widehat{\omega}_g$ and $\pi \cdot \ori_{\rho_g}$ are strict cocycles
representing the class $\theta([\omega_g])$. Both these cocycles are continuous on $\mathcal{T}$, hence we may apply~\cite[Proposition 3.1]{BurgerIozzi} to conclude that
they coincide on $\Lambda(\Gamma)^3$. But $S$ is closed, hence $\Lambda(\Gamma)=\partial\widetilde{S}$ and thus $\widehat{\omega}_g=\pi\cdot \ori_{\rho_g}$ on 
the whole $(\partial \widetilde{S})^3$. This means that every ideal triangle in $\widetilde{S}$ has area equal to $\pi$. Now~\cite[Th\'eor\`eme 3.11]{barge1988surfaces} implies that $\widetilde{S}$, hence $S$, has constant curvature, whence the conclusion.

Finally, let $g_1$ and $g_2$ be  negatively curved Riemannian metrics on $S$ such that 
$[\omega_{g_1}]=[\omega_{g_2}]$ in $H^2_b(S)$. We have shown in the previous section that $e_b(\rho_{g_1})=e_b(\rho_{g_2})$, hence
$\theta([\omega_{g_1}])=-2\pi e_b(\rho_{g_1})$ if and only if $\theta([\omega_{g_2}])=-2\pi e_b(\rho_{g_2})$. This shows that $g_1$ is hyperbolic if and only if $g_2$ is hyperbolic,
and concludes the proof of the theorem.
\end{proof}

\begin{rem}\label{constant:rem}
It readily follows from the definitions that, if $\lambda>0$ is a rescaling factor and $g$ is any pinched negatively curved metric on $S$, then $[\omega_{\lambda\cdot g}]=\lambda^2 [\omega_g]$. Therefore, Theorem~\ref{eulerhyp} implies that  $\theta([\omega_g])$ is a multpile of $e_b(\rho_g)$ if and only if the metric $g$ has constant curvature.
\end{rem}

\begin{rem}
We conjecture that Theorem~\ref{recognizing:hyperbolic} should hold also under the weaker hypothesis that the limit sets of the holonomies of $g_1$ and $g_2$ are both  equal to the whole
$\partial \widetilde{S}$ (in which case, one says that $g_1$ and $g_2$ are \emph{of the first kind}). In fact, the key ingredient for the proof of Theorem~\ref{recognizing:hyperbolic} is~\cite[Th\'eor\`eme 3.11]{barge1988surfaces}, which in turn 
easily follows from Barge--Ghys' results on the injectivity of the map $\Psi_g\colon \Omega_b^2(S)\to H^2_b(S)$ associating to every bounded $2$-form
$\alpha\in \Omega_b^2(S)$ a bounded cohomology class in $H^2_b(S)$. The injectivity of the map $\Psi_g$ was proven by Barge and Ghys for negatively curved structures on closed surfaces, and has been recently extended in~\cite{DFH} to the case of non-compact hyperbolic surfaces of the first kind. In order to prove our conjecture, it would be sufficient
to extend the main result of~\cite{DFH} to the case of (pinched) non-constant negative curvature. However, this extension does not seem straightforward since the aforementioned result  heavily relies  on the Helgason-Fourier transform in the hyperbolic plane, which is not available in variable curvature.

On the contrary, Theorem~\ref{recognizing:hyperbolic} cannot hold for structures of the second kind: indeed, the bounded class associated to the area form of a surface $S$
only sees the geometry of the compact core of $S$ (see e.g.~\cite[Proposition 3.1]{DFH}). Therefore, 
if $g_1$ is a hyperbolic structure on $S$ with some funnel (i.e.~some infinite-volume end diffeomorphic to a cylinder), then we may slightly perturb $g_1$ in a funnel, thus obtaining a 
structure $g_2$ with non-constant curvature such that the convex core of $(S,g_1)$ and $(S,g_2)$ coincide. We would then have $[\omega_{g_1}]=[\omega_{g_2}]$ with $g_1$ hyperbolic and $g_2$ non-hyperbolic.
\end{rem}

\subsection{Hyperbolic structures minimize the norm of the bounded area class}
Thanks to Theorem~\ref{hyperbolic:equal}, for every oriented surface $S$ admitting a complete hyperbolic structure it makes sense to define the element $\omega_S\in H^2_b(S)$ 
by requiring that $\omega_S=[\omega_g]$ for some (and then any) complete hyperbolic structure $g$ on $S$. 
In order to prove Theorem~\ref{extremal:thm}, we will make use of the following result (which readily follows from~\cite{BIWsurvey}), which provides a useful characterization of the class $\omega_S$: 

\begin{thm}\label{extremality1}
Let $S$ be a closed oriented surface with negative Euler characteristic, and let $\alpha\in H^2_b(S)$ be such that
$$
\langle \alpha, [S]\rangle= 2\pi \cdot |\chi(S)|\ .
$$
Then $\|\alpha\|_\infty \geq \pi$ and, if the equality $\|\alpha\|_\infty=\pi$ holds,
then $\alpha=\omega_S$.
\end{thm}
\begin{proof}
Let us fix a hyperbolic structure $g$ on $S$, i.e.~an identification between $S$ and $\mathbb{H}^2/\Gamma$, where $\Gamma\cong \pi_1(S)$ is a lattice
in $G=\isom^+(\mathbb{H}^2)$. Let then
 $t_b\colon H^2_b(\Gamma)\to \mathbb{R}$ be the functional defined in~\cite[Section 4.3]{BIWsurvey} (for the interested reader, 
for every $\beta\in H^2_b(\Gamma)$ the value of $t_b(\beta)$ is determined by the equality $T_b(\beta)=t_b(\beta)\cdot \kappa_G^b$, where 
$T_b\colon H^2_b(\Gamma)\to H^2_{cb}(G)$ is the transfer map, and $\kappa_G^b$ is the bounded K\"ahler class of $G$). Then~\cite[Theorem 4.9]{BIWsurvey}
implies that
$$
t_b(\theta(\alpha))=\frac{\langle \alpha, [S]\rangle}{|\chi(S)|}=2\pi \ .
$$
By~\cite[Theorem 4.14]{BIWsurvey}, this implies that
$$
\|\alpha\|_\infty=\|\theta(\alpha)\|_\infty\geq \frac{t_b(\theta(\alpha))}{2}=\pi\ .
$$
Moreover, if the equality $\|\alpha\|_\infty=\pi$ holds, then 
$\theta(\alpha)$ is proportional to the pull-back of $\kappa_G^b$ to $\Gamma$, which, by~\cite[Proposition 4.7]{BIWsurvey},
is equal to the bounded Euler class of the action $\rho_g$ of $\Gamma$ on $\partial \mathbb{H}^2$. Since $g$ is hyperbolic, we know that $\theta([\omega_g])=-2\pi e_b(\rho_g)$,
hence $\theta(\alpha)$ is proportional to $\theta([\omega_g])$, and $\alpha$ is proportional to $[\omega_g]$. But
$$
\langle \alpha, [S]\rangle= 2\pi \cdot |\chi(S)|=\text{Area}(S,g)=\langle [\omega_g], [S]\rangle\ ,
$$
hence the proportionality constant is equal to $1$ and $\alpha=[\omega_g]$.
\end{proof}

We are now ready to prove Theorem~\ref{extremal:thm}.
Let $g_1$ be a negatively curved metric on $S$ whose area is equal to $2\pi |\chi(S)|$.
Then $$\langle [\omega_{g_1}], [S]\rangle=2\pi |\chi(S)|\ ,$$
hence Theorem~\ref{extremality1}
implies that $\|[\omega_{g_1}]\|_\infty\geq \pi$,  and that the equality $\|[\omega_{g_1}]\|_\infty= \pi$ holds if and only if 
$[\omega_{g_1}]=[\omega_{g_2}]$ for some (and then any) hyperbolic metric $g_2$ on $S$. By Theorem~\ref{eulerhyp}, the condition
$[\omega_{g_1}]=[\omega_{g_2}]$ implies in turn that
$g_1$ is itself hyperbolic, thus concluding the proof of Theorem~\ref{extremal:thm}.

\subsection{Proof of Theorem~\ref{MCG:action}}
Theorem~\ref{MCG:action} is  an almost immediate consequence of Theorem~\ref{recognizing:hyperbolic}, together with a result from~\cite{bargagnati2024action}
regarding the action of the mapping class group of a closed surface on its bounded cohomology. In fact, let $S$
be a closed oriented surface, endowed with a negatively curved metric $g$ with non-constant curvature. 
By Proposition~\ref{properties:c}, the bounded class
$\theta([\omega_g])$ admits a representative in $\mathcal{L}((\partial \widetilde{S})^3)$ which is continuous on $\mathcal{T}$. Therefore, \cite[Theorem 4]{bargagnati2024action}
implies that the orbit of $[\omega_g]$ under the mapping class group of $S$ either coincides with $[\omega_g]$, in the case when $[\omega_g]$ is a multiple
of the bounded Euler class, or spans an infinite dimensional subspace of $H^2_b(\Gamma)$, in all the other cases.
But the curvature of $g$ is non-constant, hence Theorem~\ref{eulerhyp} ensures that
$[\omega_g]$ is not a multiple of the bounded Euler class. The conclusion follows.

\bibliography{math_papers}

\bibliographystyle{alpha}

\end{document}